\documentclass[a4paper,11pt]{amsart}
\usepackage{amssymb,amscd,amsxtra,xypic,longtable}
\usepackage[all]{xy}
\usepackage{color}
\usepackage{array}
\usepackage{braket}
\usepackage{url}

\usepackage[pdftex]{hyperref}
\hypersetup{%
bookmarksnumbered=true,%
colorlinks=true,%
pagebackref=true, %
pdftitle={},%
pdfauthor={}}

\theoremstyle{plain}
\newtheorem{Thm}{Theorem}[section]
\newtheorem{Lem}[Thm]{Lemma}

\newtheorem{Prop}[Thm]{Proposition}

\theoremstyle{definition}
\newtheorem{Def}[Thm]{Definition}
\newtheorem{Def-Lem}[Thm]{Definition-Lemma}
\newtheorem{Conj}[Thm]{Conjecture}

\newtheorem{Rem}[Thm]{Remark}

\newtheorem*{Ack}{Acknowledgments}

\newtheorem{Question}[Thm]{Question}

\newcommand{\Spec}{\operatorname{Spec}}

\newcommand{\rank}{\operatorname{rank}}

\newcommand{\Cl}{\operatorname{Cl}}

\newcommand{\Pic}{\operatorname{Pic}}
\newcommand{\Int}{\operatorname{Int}}
\newcommand{\bNE}{\operatorname{\overline{NE}}}
\newcommand{\Bs}{\operatorname{Bs}}

\newcommand{\ord}{\operatorname{ord}}

\newcommand{\wprod}{\operatorname{wp}}

\newcommand{\inwt}{\mathbf{w}_{\operatorname{in}}}
\newcommand{\wt}{\mathbf{w}}

\newcommand{\ivr}{\operatorname{ivr}}

\newcommand{\mbA}{\mathbb{A}}

\newcommand{\mbC}{\mathbb{C}}

\newcommand{\mbP}{\mathbb{P}}
\newcommand{\mbQ}{\mathbb{Q}}

\newcommand{\mbZ}{\mathbb{Z}}

\newcommand{\mcB}{\mathcal{B}}
\newcommand{\mcC}{\mathcal{C}}

\newcommand{\mcH}{\mathcal{H}}

\newcommand{\mcL}{\mathcal{L}}

\newcommand{\msp}{\mathsf{p}}

\newcommand{\ratmap}{\dashrightarrow}

\makeatletter
\def\imod#1{\allowbreak\mkern10mu({\operator@font mod}\,\,#1)}
\makeatother

\title[Birationally superrigid Fano 3-folds]{Birationally superrigid Fano 3-folds \\ of codimension 4}
\author{Takuzo~Okada}
\address{Department of Mathematics, Faculty of Science and Engineering, Saga University, Saga 840-8502 Japan}
\email{okada@cc.saga-u.ac.jp}
\date{}

\begin{document}

\begin{abstract}
We determine birational superrigidity for a quasi-smooth prime Fano $3$-fold of codimension $4$ with no projection centers.  
In particular we prove birational superrigidity for Fano $3$-folds of codimension $4$ with no projection centers which are recently constructed by Coughlan and Ducat.
We also pose some questions and a conjecture regarding the classification of birationally superrigid Fano $3$-folds.
\end{abstract}

\maketitle


\section{Introduction} \label{sec:intro}

A {\it prime Fano $3$-fold} is a normal projective $\mbQ$-factorial $3$-fold $X$ with only terminal singularities such that $-K_X$ is ample and the class group $\Cl (X) \cong \mbZ$ is generated by $-K_X$.
For such $X$, there corresponds the anticanonical graded ring
\[
R (X, -K_X) = \bigoplus_{m \in \mbZ_{\ge 0}} H^0 (X, - m K_X),
\]
and by choosing generators we can embed $X$ into a weighted projective space.
By the {\it codimension} of $X$ we mean the codimension of $X$ in the weighted projective space.
Based on the analysis by Alt{\i}nok, Brown, Iano-Fletcher, Kasprzyk, Prokhorov, Reid, etc. (see for example \cite{ABR}), there is a database \cite{GRDB} of numerical data (such as Hilbert series) coming from graded rings that can be the anticanonical graded ring of a prime Fano $3$-fold.
Currently it is not a classification, but it serves as an overlist, meaning that the  anticanonical graded ring of a prime Fano $3$-fold appears in the database.

The database contains a huge number of candidates, which suggests difficulty of biregular classification of Fano $3$-folds.
The aim of this paper is to shed light on the classification of birationally superrigid Fano 3-folds.
Here, a Fano $3$-fold of Picard number $1$ is said to be {\it birationally superrigid} if any birational map to a Mori fiber space is biregular.
We remark that, in \cite{AO}, a possible approach to achieving birational classification of Fano $3$-folds is suggested by introducing notion of {\it solid Fano $3$-folds}, which are Fano $3$-folds not birational to neither a conic bundles nor a del Pezzo fibration.

Up to codimension $3$, we have satisfactory results on the classification of quasi-smooth prime Fano $3$-folds: the classification is completed in codimensions $1$ and $2$ (\cite{IF}, \cite{CCC}, \cite{Al}) and in codimension $3$ the existence is known for all $70$ numerical data in the database.
Moreover birational superrigidity of quasi-smooth prime Fano $3$-folds of codimension at most $3$ has been well studied as well (see \cite{IM}, \cite{CPR}, \cite{CP}, \cite{OkadaI}, \cite{AZ}, \cite{AO}, and see also \cite{OkadaII}, \cite{OkadaIII} for solid cases in codimension $2$).

For quasi-smooth prime Fano $3$-folds of codimension $4$, there are 145 candidates of numerical data in \cite{GRDB}.
In \cite{BKR}, existence for 116 data is proved, where the construction is given by birationally modifying a known variety.
This process is called unprojection and, as a consequence, a constructed Fano $3$-fold corresponding to each of the 116 data admits a Sarkisov link to a Mori fiber space, hence it is not birationally superrigid. 
The 116 families of Fano $3$-folds are characterized as those that possesses a singular point which is so called a type $\mathrm{I}$ projection center (see \cite{BKR} for datails).
There are other types of projection centers (such as types $\mathrm{II}_1, \dots, \mathrm{II}_7, \mathrm{IV}$ according to the database \cite{GRDB}).
Through the known results in codimensions $1$, $2$ and $3$, we can expect that the existence of a projection center violates birational superrigidity.
Therefore it is natural to consider prime Fano $3$-folds without projection centers for the classification of birational superrigid Fano $3$-folds (see also the discussion in Section \ref{sec:discussion}).

According to the database \cite{GRDB}, there are $5$ candidates of quasi-smooth prime Fano $3$-folds of codimension $4$ with no projection centers.
Those are identified by database numbers $\# 25$, $\#166$, $\# 282$, $\# 308$ and $\# 29374$.
Among them, $\# 29374$ corresponds to smooth prime Fano $3$-folds of degree $10$ embedded in $\mbP^7$, and it is proved in \cite{DIM} that they are not birationally superrigid (not even birationally rigid, a weaker notion than superrigidity).
Recently Coughlan and Ducat \cite{CD} constructed many prime Fano $3$-folds including those corresponding to $\# 25$ and $\# 282$ and we sometimes refer to these varieties as {\it cluster Fano $3$-folds}.
There are two constructions, $\mathsf{G}^{(4)}_2$ and $\mathsf{C}_2$ formats (see \cite[Section 5.6]{CD} for details and see Section \ref{sec:proof282} for concrete descriptions) for $\# 282$ and they are likely to sit in different components of the Hilbert scheme.

\begin{Thm} \label{mainthm1}
Let $X$ be a quasi-smooth prime Fano $3$-fold of codimension $4$ and of numerical type $\# 282$ which is constructed in either $\mathsf{G}^{(4)}_2$ format or $\mathsf{C}_2$ format.
If $X$ is constructed in $\mathsf{C}_2$ format, then we assume that $X$ is general.
Then $X$ is birationally superrigid.
\end{Thm}

For the remaining three candidates $\# 25$, $\# 166$ and $\# 282$, we can prove birational superrigidity in a stronger manner; we are able to prove birational superrigidity for these $3$ candidates by utilizing only numerical data.
Here, by numerical data for a candidate Fano $3$-fold $X$, we mean the weights of the weighted projective space, degrees of the defining equations, the anticanonical degree $(-K_X)^3$ and the basket of singularities of $X$ (see Section \ref{sec:proofnum}).
Note that we do not know the existence of Fano $3$-folds for $\# 166$ and $\# 308$.

\begin{Thm} \label{mainthm2}
Let $X$ be a quasi-smooth prime Fano $3$-fold of codimension $4$  and of numerical type $\# 25$, $\# 166$ or $\#308$.
Then $X$ is birationally superrigid. 
\end{Thm}

\begin{Ack}
The author would like to thank Stephen Coughlan for giving me fruitful information on cluster Fano 3-folds.
He is partially supported by JSPS KAKENHI Grant Number JP18K03216.
\end{Ack}

\section{Birational superrigidity}

\subsection{Basic properties}

Throughout this subsection, we assume that $X$ is a Fano $3$-fold of Picard number $1$, that is, $X$ is a normal projective $\mbQ$-factorial $3$-fold such that $X$ has only terminal singularities, $-K_X$ is ample and $\rank \Pic (X) = 1$.

\begin{Def}
We say that $X$ is {\it birationally superrigid} if any birational map $\sigma \colon X \ratmap Y$ to a Mori fiber space $Y \to T$ is biregular.
\end{Def}

By an {\it extremal divisorial extraction} $\varphi \colon (E \subset Y) \to (\Gamma \subset X)$, we mean an extremal divisorial contraction $\varphi \colon Y \to X$ from a normal projective $\mbQ$-factorial variety $Y$ with only terminal singularities such that $E$ is the $\varphi$-exceptional divisor and $\Gamma = \varphi (E)$.

\begin{Def}
Let $\mcH \sim_{\mbQ} -n K_X$ be a movable linear system, where $n$ is a positive integer.
A {\it maximal singularity} of $\mcH$ is an extremal extraction $\varphi \colon (E \subset Y) \to (\Gamma \subset X$ such that
\[
c (X, \mcH) = \frac{a_E (K_X)}{m_E (\mcH)} < \frac{1}{n},
\]
where
\begin{itemize}
\item $c (X, \mcH) := \max \{\, \lambda \mid \text{$K_X + \lambda \mcH$ is canonical}\,\}$ is the {\it canonical threshold} of $(X, \mcH)$,
\item $a_E (K_X)$ is the discrepancy of $K_X$ along $E$, and
\item $m_E (\mcH)$ is the multiplicity along $E$ of the proper transform $\varphi_*^{-1} \mcH$ on $Y$.
\end{itemize}
We say that an extremal divisorial extraction is a {\it maximal singularity} if if there exists a movable linear system $\mcH$ such that the extraction is a maximal singularity of $\mcH$.
A subvariety $\Gamma \subset X$ is called a {\it maximal center} if there is an maximal singularity $Y \to X$ whose center is $\Gamma$. 
\end{Def}

\begin{Thm}
If $X$ admits no maximal center, then $X$ is birationally superrigid.
\end{Thm}

For a proof of birational superrigidity of a given Fano $3$-fold $X$ of Picard number $1$, we need to exclude each subvariety of $X$ as a maximal center.
In the next subsection we will explain several methods of exclusion under a relatively concrete setting.
Here we discuss methods of excluding terminal quotient singular points in a general setting.

For a terminal quotient singular point $\msp \in X$ of type $\frac{1}{r} (1, a, r-a)$, where $r$ is coprime to $a$ and $0 < a < r$, there is a unique extremal divisorial extraction $\varphi \colon (E \subset Y) \to (\msp \in X)$, which is the weighted blowup with weight $\frac{1}{r} (1, a, r-a)$, and we call it the {\it Kawamata blowup} (see \cite{Kawamata} for details). 
The integer $r > 1$ is called the {\it index} of $\msp \in X$.
For the Kawamata blowup $\varphi \colon (E \subset Y) \to (\msp \in X)$, we have $K_Y = \varphi^*K_X + \frac{1}{r} E$ and 
\[
(E^3) = \frac{r^2}{a (r-a)}.
\]
For a divisor $D$ on $X$, the {\it order} of $D$ along $E$, denote by $\ord_E (D)$, is defined to be the coefficient of $E$ in $\varphi^*D$.

We first explain the most basic method.

\begin{Lem}[{\cite[Lemma 5.2.1]{CPR}}] \label{lem:exclsingNE}
Let $\msp \in X$ a terminal quotient singular point and $\varphi \colon (E \subset Y) \to (\msp \in X)$ the Kawamata blowup.
If $(-K_Y)^2 \notin \Int \bNE (Y)$, then $\msp$ is not a maximal center.
\end{Lem}

For the application of the above lemma, we need to find a nef divisor on $Y$.
The following result, which is a slight generalization of \cite[Lemma 6.6]{OkadaII}, is useful.

\begin{Lem} \label{lem:findnef}
Let $\msp \in X$ be a terminal quotient singular point and $\varphi \colon (E \subset Y) \to (\msp \in X)$ the Kawamata blowup.
Assume that there are effective Weil divisors $D_1, \dots, D_k$ such that the intersection $D_1 \cap \cdots \cap D_k$ does not contain a curve through $\msp$.
We set
\[
e := \min \{\ord_E (D_i)/n_i \mid 1 \le i \le k\},
\]
where $n_i$ is the positive rational number such that $D_i \sim_{\mbQ} - n_i K_X$. 
Then $- \varphi^* K_X - \lambda E$ is a nef divisor for $0 \le \lambda \le e$.
\end{Lem}

\begin{proof}
We may assume $e > 0$, that is, $D_i$ passes through $\msp$ for any $i$.
For an effective divisor $D \sim_{\mbQ} - n K_X$, we call $\ord_E (D)/n$ the vanishing ratio of $D$ along $E$.
For $1 \le i \le k$, we choose a component of $D_i$, denoted $D'_i$, which has maximal vanishing ratio along $E$ among the components of $D_i$.
Clearly we have $D'_1 \cap \cdots \cap D'_k$ does not contain a curve through $\msp$ and we have
\[
e' := \min \{ \ord_E (D'_i)/n'_i \mid 1 \le i \le k\} \ge e,
\]
where $n'_i \in \mbQ$ is such that $D'_i \sim_{\mbQ} - n'_i K_X$.
Since $D'_1, \dots, D'_k$ are prime divisors, we can apply \cite[Lemma 6.6]{OkadaII} and conclude that $-\varphi^*K_X - e' E$ is nef.
Then so is $-\varphi^*K_X - \lambda E$ for any $0 \le \lambda \le e'$ since $-\varphi^*K_X$ is nef, and the proof is completed.
\end{proof}

We have another method of exclusion which can be sometimes effective when Lemma \ref{lem:exclsingNE} is not applicable.

\begin{Lem} \label{lem:exclsinginfini}
Let $\msp \in X$ be a terminal quotient singular point and $\varphi \colon (E \subset Y) \to (\msp \in X)$ the Kawamata blowup.
Suppose that there exists an effective divisor $S$ on $X$ passing through $\msp$ and a linear system $\mcL$ of divisors on $X$ passing through $\msp$ with the following properties.
\begin{enumerate}
\item $S \cap \Bs \mcL$ does not contain a curve passing through $\msp$, and
\item For a general member $L \in \mcL$, we have
\[
(- K_Y \cdot \tilde{S} \cdot \tilde{L}) \le 0,
\]
where $\tilde{S}, \tilde{L}$ are the proper transforms of $S, L$ on $Y$, respectively.
\end{enumerate}
Then $\msp$ is not a maximal center.
\end{Lem}

\begin{proof}
We write $D \sim - n K_X$.
Write $S = \sum m_i S_i + T$, where $m_i > 0$, $S_i$ is a prime divisor and $T$ is an effective divisor which does not pass through $\msp$.
We have $T \sim - l K_X$ for some $l \ge 0$ and
\[
(-K_Y \cdot \tilde{T} \cdot \tilde{L}) =  n l (-K_X)^3 \ge 0.
\]
Since
\[
0 \ge (-K_Y \cdot \tilde{S} \cdot \tilde{L}) 
= \sum m_i (-K_Y \cdot \tilde{S}_i \cdot \tilde{L}) + (-K_Y \cdot \tilde{T} \cdot \tilde{L}),
\]
there is a component $S_i$ for which $(-K_Y \cdot \tilde{S}_i \cdot \tilde
{L}) \le 0$.
Since $\msp \in S_i \cap \Bs \mcL \subset S \cap \Bs \mcL$, we may assume that $S$ is a prime divisor by replacing $S$ by $S_i$.

Write $\mcL = \{L_{\lambda} \mid \lambda \in \mbP^1\}$.
For $\lambda \in \mbP^1$, we write $S \cdot L_{\lambda} = \sum_i c_i C_{\lambda, i}$, where $c_i \ge 0$ and $C_{\lambda,i}$ is an irreducible and reduced curve on $X$.
For a curve or a divisor $\Delta$ on $X$, we denote by $\tilde{\Delta}$ its proper transform on $Y$.
Then,
\[
\tilde{S} \cdot \tilde{L}_{\lambda} = \sum_i c_i \tilde{C}_{\lambda,i} + \Xi,
\]
where $\Xi$ is an effective $1$-cycle supported on $E$.
Since any component of $\Xi$ is contracted by $\varphi$ and $-K_Y$ is $\varphi$-ample, we have $(-K_Y \cdot \Xi) \ge 0$.
Thus, for a general $\lambda \in \mbP^1$, we have
\[
0 \ge (-K_Y \cdot \tilde{S} \cdot \tilde{L}_{\lambda}) \ge \sum_i c_i (-K_Y \cdot \tilde{C}_{\lambda,i}).
\]
It follows that $(-K_Y \cdot \tilde{C}_{\lambda,i}) \le 0$ for some $i$.
We choose such a $\tilde{C}_{\lambda,i}$ and denote it as $\tilde{C}^{\circ}_{\lambda}$. 
By the assumption (1), the set 
\[
\{\tilde{C}^{\circ}_{\lambda} \mid \text{$\lambda \in \mbP^1$ is general}\}
\]
consists of infinitely many distinct curves.
We have $(-K_Y \cdot \tilde{C}^{\circ}_{\lambda}) \le 0$ by the construction.
We see that $(E \cdot \tilde{C}_{\lambda}^{\circ}) > 0$ since $\tilde{C}_{\lambda}$ is the proper transform of a curve passing through $\msp$.
Therefore $\msp$ is not a maximal center by \cite[Lemma 2.20]{OkadaII}.
\end{proof}

\subsection{Fano varieties in a weighted projective space}

Let $\mbP = \mbP (a_0,\dots,a_n)$ be a weighted projective space with homogeneous coordinates $x_0,\dots,x_n$ of $\deg x_i = a_i$.
We assume that $\mbP$ is well formed, that is, 
\[
\gcd \{\, a_i \mid 0 \le i \le n, i \ne j \,\} = 1
\]
for $j = 0, 1, \dots, n$.
Throughout the present subsection, let $X \subset \mbP$ be a normal projective $3$-fold defined by the equations
\[
F_1 = F_2 = \cdots = F_N = 0,
\]
where $F_i \in \mbC [x_0,\dots,x_{n+1}]$ is a homogeneous polynomial of degree $d_i$ with respect to the grading $\deg x_i = a_i$.

\begin{Def}
We say that $X$ is {\it quasi-smooth} if the affine cone
\[
(F_1 = F_2 = \cdots = F_N = 0) \subset \mbA^{n+1} = \Spec \mbC [x_0,\dots,x_n]
\]
is smooth outside the origin. 
\end{Def}

In the following we assume that $X$ is a quasi-smooth prime Fano $3$-fold.
For $0 \le i \le n$, we define $\msp_{x_i} = (0\!:\!\cdots\!:\!1\!:\!\cdots\!:\!0) \in \mbP$, where the unique $1$ is in the $(i+1)$st position, and we define $D_i = (x_i = 0) \cap X$ which is a Weil divisor such that $D_i \sim - a_i K_X$.

\begin{Lem} \label{lem:exclcurve}
If $(-K_X)^3 \le 1$, then no curve on $X$ is a maximal center.
\end{Lem}

\begin{proof}
The same proof of \cite[Lemma 2.1]{AO} applies in this setting without any change.
\end{proof}

\begin{Lem} \label{lem:exclnonsingpt}
Assume that $a_0 \le a_1 \le \dots \le a_n$.
If $a_{n-1} a_n (-K_X)^3 \le 4$, then no nonsingular point of $X$ is a maxinal center.
\end{Lem}

\begin{proof}
The proof is almost identical to that of \cite[Lemma 2.6]{AO}.
\end{proof}

\begin{Def}
Let $\mcC \subset \{x_0,\dots,x_n\}$ be a set of homogeneous coordinates.
We define
\[
\begin{split}
\Pi (\mcC) &:= \bigcap_{z \in \mcC} (z = 0) \subset \mbP, \\
\Pi_X (\mcC) &:= \Pi (\mcC) \cap X \subset X.
\end{split}
\]
We also denote
\[
\Pi (\mcC) = \Pi (x_{i_1}, \dots, x_{i_m}), \quad
\Pi_X (\mcC) = \Pi_X (x_{i_1}, \dots, x_{i_m}),
\]
when $\mcC = \{x_{i_1}, \dots, x_{i_m}\}$.
\end{Def}

\begin{Lem} \label{lem:exclcri1/2pt}
Let $\msp \in X$ be a singular point of type $\frac{1}{2} (1,1,1)$ and let
\[
b := \max \{\, a_i \mid 0 \le i \le n, \text{$a_i$ is odd}\,\}.
\]
If $2 b (-K_X)^3 \le 1$, then $\msp$ is not a maximal center. 
\end{Lem}

\begin{proof}
Let $\mcC = \{ x_{i_1}, \dots, x_{i_m}\}$ be the set of homogeneous coordinates of odd degree.
The set $\Pi_X (\mcC) = D_{i_1} \cap \cdots \cap D_{i_m}$ consists of singular points since $X$ is quasi-smooth and has only terminal quotient singularities (which are isolated), 
In particular $\Pi_X (\mcC)$ is a finite set of points.
Let $\varphi \colon (E \subset Y) \to (\msp \in X)$ be the Kawamata blowup. 
Then $\ord_E (D_{i_j}) \ge 1/2$ since $2 D_{i_j}$ is a Cartier divisor passing through $\msp$ and thus $- b \varphi^* K_X - \frac{1}{2} E$ is nef by Lemma \ref{lem:findnef}.
We have
\[
(-b \varphi^* K_X - \frac{1}{2} E) \cdot (-K_Y)^2 = b (-K_X)^3 - \frac{1}{2} \le 0.
\]
This shows that $(-K_Y)^2 \notin \bNE (Y)$ and $\msp$ is not a maximal center by Lemma \ref{lem:exclsingNE}.
\end{proof}

\begin{Def}
Let $\msp = \msp_{x_k} \in X$ be a terminal quotient singular point of type $\frac{1}{a_k} (1, c, a_k-c)$ for some $c$ with $1 \le c \le a_k/2$.
We define
\[
\ivr_{\msp} (\mcC) := \min_{1 \le j \le m} \left\{\frac{\overline{a_{i_j}}}{a_{i_j} a_k} \right\},
\]
where $\mcC = \{x_{i_1}, \dots, x_{i_m}\}$ and $\overline{a_{j_i}}$ is the integer such that $1 \le \overline{a_{j_i}} \le a_k$ and $\overline{a_{j_i}}$ is congruent to $a_{j_i}$ modulo $a_k$, and call it the {\it initial vanishing ratio} of $\mcC$ at $\msp$.
\end{Def}

\begin{Def}
For a terminal quotient singularity $\msp$ of type $\frac{1}{r} (1,a,r-a)$, we define
\[
\wprod (\msp) := a (r-a),
\]
and call it the {\it weight product} of $\msp$.
\end{Def}

\begin{Lem} \label{lem:exclcrisingpt}
Let $\msp = \msp_{x_k} \in X$ be a terminal quotient singular point.
Suppose that there exists a subset $\mcC \subset \{x_0,\dots,x_n\}$ satisfying the following properties.
\begin{enumerate}
\item $\msp \in \Pi_X (\mcC)$, or equivalently $x_k \notin \mcC$.
\item $\Pi_X (\mcC \cup \{x_k\}) = \emptyset$.
\item $\ivr_{\msp} (\mcC) \ge \wprod (\msp) (-K_X)^3$.
\end{enumerate} 
Then $\msp$ is not a maximal center.
\end{Lem}

\begin{proof}
We write $\mcC = \{x_{i_1}, \dots, x_{i_m}\}$.
We claim that $\Pi_X (\mcC) = D_{i_1} \cap \cdots \cap D_{i_m}$ is a finite set of points.
Indeed, if it contains a curve, then $\Pi_X (\mcC \cup \{x_k\}) = \Pi_X (\mcC) \cap D_k \ne \emptyset$ since $D_k$ is an ample divisor on $X$.
This is impossible by the assumption (2).
Note that we have $\ord_E (D_{i_j}) \ge \overline{a_{i_j}}/a_k$ (cf.\ \cite[Section 3]{AO}) so that
\[
e := \min \{\, \ord_E (D_{i_j})/a_{i_j} \mid 1 \le j \le m \,\} \ge \ivr_{\msp} (\mcC).
\]
By Lemma \ref{lem:findnef}, $- \varphi^* K_X - \ivr_{\msp} (\mcC) E$ is nef and we have
\[
(- \varphi^* K_X - \ivr_{\msp} (\mcC) E) (-K_Y)^2 = (-K_X)^3 - \frac{\ivr_{\msp} (\mcC)}{\wprod (\msp)} \le 0
\]
by the assumption (3).
Therefore $(-K_Y)^2 \notin \bNE (Y)$ and $\msp$ is not a maximal center.
\end{proof}

Let $\msp \in X$ be a singular point such that it can be transformed to $\msp_{x_k}$ by a change of coordinates.
For simplicity of the description we assume $\msp = \msp_{x_0}$ and we set $r = a_0 > 1$.
Let $\varphi \colon (E \subset Y) \to (\msp \in X)$ be the Kawamata blowup.
We explain a systematic way to estimate $\ord_E (x_i)$ for coordinates $x_i$ and also an explicit description of $\varphi$.
It is a consequence of the quasi-smoothness of $X$ that after re-numbering the defining equation we can write
\[
F_l = \alpha_l x_k^{m_l} x_{i_l} + (\text{other terms}), \quad \text{for $1 \le l \le n-3$},
\]
where $\alpha_l \in \mbC \setminus \{0\}$, $m_l$ is a positive integer and $x_0, x_{i_1}, \dots, x_{i_{n-3}}$ are mutually distinct so that by denoting the other $3$ coordinates as $x_{j_1}, x_{j_2}, x_{j_3}$ we have
\[
\{x_0, x_{i_1},\dots,x_{i_{n-3}}, x_{j_1}, x_{j_2}, x_{j_3} \} = \{x_0,\dots,x_n\}.
\]
In this case we can choose $x_{j_1}, x_{j_2}, x_{j_3}$ as local orbi-coordinates of $X$ at $\msp$ and the singular point $\msp$ is of type 
\[
\frac{1}{r} (a_{j_1}, a_{j_2}, a_{j_3}).
\]

\begin{Def}[{\cite[Definitions 3.6, 3.7]{AO}}]
For an integer $a$, we denote by $\bar{a}$ the positive inter such that $\bar{a} \equiv a \pmod{r}$ and $0 < \bar{a} \le r$.
We say that
\[
\wt (x_1,\dots,x_n) = \frac{1}{r} (b_1,\dots,b_n)
\]
is an {\it admissible weight} at $\msp$ if $b_i \equiv a_i \pmod{r}$ for any $i$.

For an admissible weight $\wt$ at $\msp$ and a polynomial $f = f (x_0,\dots,x_n)$, we denote by $f^{\wt}$ the lowest weight part of $f$, where we assume that $\wt (x_0) = 0$.

We say that an admissible weight $\wt$ at $\msp$ satisfies the {\it KBL condition} if $x_0^{e_l} x_{i_l} \in F_l^{\wt}$ for $1 \le l \le n-3$ and
\[
(b_{j_1}, b_{j_2}, b_{j_3}) = (\overline{a_{j_1}}, \overline{a_{j_2}}, \overline{a_{j_3}}).
\]
\end{Def}

Let $\wt (x_1,\dots,x_n) = \frac{1}{r} (b_1,\dots,b_n)$ be an admissible weight at $\msp$ satisfying the KBL condition.
We denote by $\Phi_{\wt} \colon Q_{\wt} \to \mbP$ at $\msp$ with weight $\wt$, and by $Y_{\wt}$ the proper transform of $X$ via $\Phi_{\wt}$.
Then the induced morphism $\varphi_{\wt} = \Phi_{\wt}|_{Y_{\wt}} \colon Y_{\wt} \to X$ coincides with the Kawamata blowup at $\msp$.
From this we see that the exceptional divisor $E$ is isomorphic to 
\[
E_{\wt} := (f_1 = \cdots = f_{n-3} = 0) \subset \mbP (b_1,\dots,b_n),
\]
where $f_l = F_l^{\wt} (1,x_1,\dots,x_n)$.
We refer readers to \cite[Section 3]{AO} for details.

\begin{Lem}[{\cite[Lemma 3.9]{AO}}] \label{lem:ordwt}
Let $\wt (x_1,\dots,x_n) = \frac{1}{r} (b_1,\dots,b_n)$ be an admissible weight at $\msp \in X$ satisfying the KBL condition.
Then the following assertions hold.
\begin{enumerate}
\item We have $\ord_E (D_i) \ge b_i/r$ for any $i$.
\item If $F_l^{\wt} = \alpha_l x_0^{m_l} x_{i_l}$, where $\alpha_i \in \mbC \setminus \{0\}$, for some $1 \le l \le n-3$, then the weight
\[
\wt' (x_1,\dots,x_n) = \frac{1}{r} (b'_1,\dots,b'_n),
\]
where $b'_j = b_j$ for $j \ne l$ and $b'_l = b_l + r$, satisfies the KBL condition.
In particular, $\ord_E (D_l) \ge (b_l+r)/r$.
\end{enumerate}
\end{Lem}

We will use the following notation for a polynomial $f = f (x_0,\dots,x_n)$.
\begin{itemize}
\item For a monomial $p = x_0^{e_0} \cdots x_n^{e_n}$, we write $p \in f$ if $p$ appears in $f$ with non-zero (constant) coefficient.
\item For a subset $\mcC \subset \{x_0,\dots,x_n\}$ and $\Pi = \Pi (\mcC)$, we denote by $f|_{\Pi}$ the polynomial in variables $\{x_0,\dots,x_n\} \setminus \mcC$ obtained by putting $x_i = 0$ for $x_i \in \mcC$ in $f$.
\end{itemize}

\begin{Rem}
We explain some consequences of quasi-smoothness, which will be frequently used in Section \ref{sec:proofnum}.
We keep the above notation and assumption.
In particular $X \subset \mbP (a_0,\dots,a_n)$ is assumed to be quasi-smooth.
\begin{enumerate}
\item Let $\mcC = \{ x_{i_1}, \dots, x_{i_m} \}$ be the coordinates such that $r := \gcd \{ a_{i_1}, \dots, a_{i_m} \} > 1$ and $a_j$ is coprime to $r$ for any $j \ne i_1, \dots, i_m$.
Then 
\[
\Sigma := \Pi_X (\{x_0,\dots,x_n\} \setminus \mcC)
\] 
is contained in the singular locus of $X$ and $X$ has a quotient singular point of index $r$ at each point of $\Sigma$.
In particular, if $X$ has only isolated singularities (e.g.\ $\dim X = 3$ and $X$ has only terminal singularities), then either $\Sigma = \emptyset$ of $\Sigma$ consists of finite set of singular points of index $r$.
\item Let $x_k$ be the coordinate such that $a_j \ne a_k$ for any $j \ne k$.
If $X$ does not contain a singular point of index $r$, then $\msp_k \notin X$, that is, a power of $x_k$ appears in one of the defining polynomials with non-zero coefficient.
\end{enumerate}
\end{Rem}

\section{Proof of birational superrigidity by numerical data} \label{sec:proofnum} \label{sec:proofnum}

We prove birational superrigidity of codimension $4$ quasi-smooth prime Fano $3$-folds with no projections by utilizing only numerical data.
The numerical data for each Fano $3$-fold will be described in the beginning of the corresponding subsection.
The Fano $3$-folds are embedded in a weighted projective $7$-space, denoted by $\mbP$, and we use the symbol $p,q,r,s,t,u,v,w$ for the homogeneous coordinates of $\mbP$.
We use the following terminologies: 
Let $X \subset \mbP$ be a codimension $4$ quasi-smooth prime Fano $3$-fold.
For a homogeneous coordinate $z \in \{p,q,\dots,w\}$,
\begin{itemize}
\item $D_z := (z = 0) \cap X$ is the Weil divisor on $X$ cut out by $z$, and
\item $\msp_z \in \mbP$ is the point at which only the coordinate $z$ does not vanish.
\end{itemize} 

Note that Theorem \ref{mainthm2} will follow from Theorems \ref{thm:25}, \ref{thm:166}, \ref{thm:308}.

\subsection{Fano $3$-folds of numerical type $\# 25$}

Let $X$ be a quasi-smooth prime Fano $3$-fold of numerical type $\# 25$, whose data consist of the following.
\begin{itemize}
\item $X \subset \mbP (2_p,5_q,6_r,7_s,8_t,9_u,10_v,11_w)$. \\[-3.5mm]
\item $(-K_X)^3 = 1/70$. \\[-3.5mm]
\item $\deg (F_1,F_2,F_3,F_4,F_5,F_6,F_7,F_8,F_9) = (16,17,18,18,19,20,20,21,22)$. \\[-3.5mm]
\item $\mcB_X = \left\{ 7 \times \frac{1}{2} (1,1,1), \frac{1}{5} (1,1,4), \frac{1}{7} (1,2,5) \right\}$.
\end{itemize}

Here the subscripts $p, q, \dots, w$ of the weights means that they are the homogeneous coordinates of the indicated degrees, and $\mcB_X$ indicates the numbers and the types of singular points of $X$. 

\begin{Thm} \label{thm:25}
Let $X$ be a quasi-smooth prime codimension $4$ Fano $3$-fold of numerical type $\#25$.
Then $X$ is birationally superrigid.
\end{Thm}

\begin{proof}
By Lemmas \ref{lem:exclcurve} and \ref{lem:exclnonsingpt}, no curve and no nonsingular point on $X$ is a maximal center.
By Lemma \ref{lem:exclcri1/2pt}, singular points of type $\frac{1}{2} (1,1,1)$ are not maximal centers.

Let $\msp$ be the singular point of type $\frac{1}{5} (1,1,4)$.
Replacing the coordinate $v$ if necessary, we may assume $\msp = \msp_q$.
We set $\mcC = \{p,s,u,v\}$.
We have
\[
\ivr_{\msp} (\mcC) = \frac{2}{35} = \wprod (\msp) (-K_X)^3.
\]
By Lemma \ref{lem:exclcrisingpt}, it remains to show that $\Pi_X := \Pi_X (\mcC \cup \{q\}) = \emptyset$.
We set $\Pi := \Pi (\mcC \cup \{q\}) \subset \mbP$ so that $\Pi_X = \Pi \cap X$.
Since $\msp_t \notin X$, one of the defining polynomials contain a power of $t$.
By looking at the degrees of $F_1, \dots, F_9$, we have $t^2 \in F_1$.
Similarly, we have $r^3 \in F_3$ and $w^2 \in F_9$ after possibly interchanging $F_3$ and $F_4$.
The monomial $t^2$ (resp.\ $r^3$) is the only monomial of degree $16$ (resp.\ $18$) consisting of the variables $r,t,w$.
The monomials $w^2$ and $t^2 r$ are the only monomials of degree $22$ consisting the variables $r,t,w$.
Hence, re-scaling $r, t, w$, we can write
\[
F_1|_{\Pi} = t^2, \quad
F_3|_{\Pi} = r^3, \quad
F_9|_{\Pi} = w^2 + \alpha t^2 r,
\]
for some $\alpha \in \mbC$.
The set $\Pi_X$ is contained in the common zero loci of the above 3 polynomials inside $\Pi$.
The equations have only trivial solution and this shows that $\Pi_X = \emptyset$.
Thus $\msp$ is not a maximal center.

Let $\msp = \msp_s$ be the singular point of type $\frac{1}{7} (1,2,5)$ and  set $\mcC = \{p,q,r\}$.
We have
\[
\ivr_{\msp} (\mcC) = \frac{1}{7} = \wprod (\msp) (-K_X)^3.
\]
By Lemma \ref{lem:exclcrisingpt}, it remains to show that $\Pi_X := \Pi_X (\mcC \cup \{s\}) = \emptyset$.
We set $\Pi := \Pi (\mcC \cup \{s\}) \subset \mbP$ so that $\Pi_X = \Pi \cap X$.
Since $\msp_t, \msp_u, \msp_v, \msp_w \notin X$, we may assume $t^2 \in F_1$, $u^2 \in F_3$, $v^2 \in F_6$ and $w^2 \in F_9$ after possibly interchanging defining polynomials of the same degree.
Then we can write
\[
F_1|_{\Pi} = t^2, \quad
F_3|_{\Pi} = u^2 + \alpha v t, \quad
F_6|_{\Pi} = v^2 + \beta w u, \quad
F_9|_{\Pi} = w^2 + \gamma t^2 r,
\] 
for some $\alpha, \beta, \gamma \in \mbC$.
This shows that $\Pi_X = \emptyset$ and thus $\msp$ is not a maximal center.
This completes the proof.
\end{proof}

\subsection{Fano $3$-folds of numerical type $\#166$}
Let $X$ be a quasi-smooth prime Fano $3$-fold of numerical type $\# 166$, whose data consist of the following.
\begin{itemize}
\item $X \subset \mbP (2_p,2_q,3_r,3_s,4_t,4_u,5_v,5_w)$. \\[-3.5mm]
\item $(-K_X)^3 = 1/6$. \\[-3.5mm]
\item $\deg (F_1,F_2,F_3,F_4,F_5,F_6,F_7,F_8,F_9) = (8,8,8,9,9,9,10,10,10)$. \\[-3.5mm]
\item $\mcB_X = \left\{ 11 \times \frac{1}{2} (1,1,1), \frac{1}{3} (1,1,2) \right\}$.
\end{itemize}

\begin{Thm} \label{thm:166}
Let $X$ be a quasi-smooth prime codimension $4$ Fano $3$-fold of numerical type $\#166$.
Then $X$ is birationally superrigid.
\end{Thm}

\begin{proof}
By Lemmas \ref{lem:exclcurve} and \ref{lem:exclnonsingpt}, no curve and no nonsingular point is a maximal center.

Let $\msp$ be a singular point of type $\frac{1}{2} (1,1,1)$.
After replacing coordinates, we may assume $\msp = \msp_p$.
We set $\mcC = \{q, r, s, t, u\}$.
We have
\[
\ivr_{\msp} (\mcC) = \frac{1}{6} = \wprod (\msp) (-K_X)^3.
\]
Moreover we have $\Pi_X (\mcC \cup \{p\}) = \emptyset$ because $X$ is quasi-smooth and it does not have a singular point of index $5$.
Thus, by Lemma \ref{lem:exclcrisingpt}, $\msp$ is not a maximal center.

Let $\msp$ be the singular point of type $\frac{1}{3} (1,1,2)$.
After replacing $r$ and $s$, we may assume $\msp = \msp_s$.
We set $\mcC = \{p,q,r\}$.
Then we have
\[
\ivr_{\msp} (\mcC) = \frac{1}{3} = \wprod (\msp) (-K_X)^3.
\]
By Lemma \ref{lem:exclcrisingpt}, it remains to show that $\Pi_X := \Pi_X (\mcC \cup \{s\}) = \emptyset$.
We set $\Pi := \Pi (\mcC \cup \{s\}) \subset \mbP$ so that $\Pi_X = \Pi \cap X$.
We have
\[
\Pi_X = (F_1|_{\Pi} = F_2|_{\Pi} = F_3|_{\Pi} = F_7|_{\Pi} = F_8|_{\Pi} = F_9|_{\Pi} = 0) \cap \Pi.
\]
We see that $F_1|_{\Pi}, F_2|_{\Pi}, F_3|_{\Pi}$ consist only of monomials in variables $t, u$, and $X$ does not have a singular point of index $4$.
Hence the equation
\[
F_1|_{\Pi} = F_2|_{\Pi} = F_3|_{\Pi} = 0
\]
implies $t = u = 0$.
Similarly, $F_7|_{\Pi}, F_8|_{\Pi}, F_9|_{\Pi}$ consist only of the monomials in variables $v, w$, and $X$ does not contain a singular point of index $5$.
Hence the equation
\[
F_7|_{\Pi} = F_8|_{\Pi} = F_9|_{\Pi} = 0
\]
implies $v = w = 0$.
It follows that $\Pi_X = \emptyset$ and $\msp$ is not a maximal center.
Therefore $X$ is birationally superrigid.
\end{proof}

\subsection{Fano $3$-folds of numerical type $\#282$}
Let $X$ be a quasi-smooth prime Fano $3$-fold of numerical type $\# 282$, whose data consist of the following.
\begin{itemize}
\item $X \subset \mbP (1_p, 6_q, 6_r, 7_s, 8_t, 9_u, 10_v, 11_w)$. \\[-3.5mm]
\item $(-K_X)^3 = 1/42$. \\[-3.5mm]
\item $\deg (F_1, F_2, F_3, F_4, F_5,F_6, F_7, F_8, F_9) = (16, 17, 18, 18, 19, 20, 20, 21, 22)$. \\[-3.5mm]
\item $\mcB = \left\{ 2 \times \frac{1}{2} (1,1,1), 2 \times \frac{1}{3} (1,1,2), \frac{1}{6} (1,1,5), \frac{1}{7} (1,1,6) \right\}$.
\end{itemize}

\begin{Prop} \label{prop:282}
Let $X$ be a quasi-smooth prime codimension $4$ Fano $3$-fold of numerical type $\#282$.
Then no curve and no point is a maximal center except possibly for the singular point of type $\frac{1}{6} (1,1,5)$.
\end{Prop}

\begin{proof}
By Lemmas \ref{lem:exclcurve}, \ref{lem:exclnonsingpt} and \ref{lem:exclcri1/2pt}, it remains to exclude singular points of type $\frac{1}{3} (1,1,2)$ and $\frac{1}{7} (1,1,6)$ as maximal centers.

Let $\msp$ be a singular point of type $\frac{1}{3} (1,1,2)$ and let $\varphi \colon (E \subset Y) \to (\msp \in X)$ be the Kawamata blowup.
We claim that $\Pi_X (p,s,t,w) = D_p \cap D_s \cap D_t \cap D_w$ is a finite set of points (containing $\msp$).
Since $X$ does not contain a singular point of index $10$, we may assume that $v^2 \in F_6$.
Then, by re-scaling $v$, we have
\[
F_6 (0,q,r,0,0,u,v,0) = v^2
\]
and this shows that $\Pi_X (p,s,t,w) = \Pi_X (p,s,t,v,w)$.
The latter set consists of singular points $\{2 \times \frac{1}{2} (1,1,2), \frac{1}{6} (1,1,5)\}$ and thus $\Pi_X (p,s,t,w)$ is a finite set of points.
We have
\[
\ord_E (D_p), \ord_E (D_s) \ge \frac{1}{3}, \quad 
\ord_E (D_t), \ord_E (D_w) \ge \frac{2}{3}.
\]
By Lemma \ref{lem:findnef}, $N := -\varphi^*K_X - \frac{1}{21} E$ is a nef divisor on $Y$ and we have $(N \cdot (-K_Y)^2) = 0$.
Thus $\msp$ is not a maximal center.

Let $\msp = \msp_s$ be the singular point of type $\frac{1}{7} (1,1,6)$ and
set $\mcC = \{p,q,r\}$.
We have
\[
\ivr_{\msp} (\mcC) = \frac{1}{7} = \wprod (\msp) (-K_X)^3.
\]
We set $\Pi := \Pi (\mcC \cup \{s\})$.
Since $\msp_t, \msp_u, \msp_v, \msp_w \notin X$, we have $t^2 \in F_1$, $w^2 \in F_9$ and we may assume $u^2 \in F_3$, $v^2 \in F_6$.
Then, by re-scaling $t, u, v, w$, we can write
\[
F_1|_{\Pi} = t^2, \ 
F_3|_{\Pi} = \alpha v t + u^2, \ 
F_6|_{\Pi} = \beta w u + v^2, \ 
F_9|_{\Pi} = w^2,
\]
where $\alpha, \beta \in \mbC$.
This shows that $\Pi_X (\mcC \cup \{s\}) = \Pi \cap X = \emptyset$.
Thus $\msp$ is not a maximal center by Lemma \ref{lem:exclcrisingpt} and the proof is completed.
\end{proof}

\subsection{Fano $3$-folds of numerical type $\#308$}
Let $X$ be a quasi-smooth prime Fano $3$-fold of numerical type $\# 308$, whose data consist of the following.
\begin{itemize}
\item $X \subset \mbP (1_p,5_q,6_r,6_s,7_t,8_u,9_v,10_w)$. \\[-3.5mm]
\item $(-K_X)^3 = 1/30$. \\[-3.5mm]
\item $\deg (F_1, F_2, F_3, F_4, F_5, F_6, F_7, F_8, F_9) = (14, 15, 16, 16, 17, 18, 18, 19, 20)$. \\[-3.5mm]
\item $\mcB_X = \left\{ \frac{1}{2} (1,1,1), \frac{1}{3} (1,1,2), \frac{1}{5} (1,2,3), 2 \times \frac{1}{6} (1,1,5) \right\}$.
\end{itemize}

\begin{Thm} \label{thm:308}
Let $X$ be a quasi-smooth prime Fano $3$-fold of numerical type $\# 308$.
Then $X$ is birationally superrigid.
\end{Thm}

\begin{proof}
By Lemmas \ref{lem:exclcurve}, \ref{lem:exclnonsingpt} and \ref{lem:exclcri1/2pt}, no curve and no nonsingular point is a maximal center and the singular point of type $\frac{1}{2} (1,1,1)$ is not a maximal center.

Let $\msp$ be the singular point of type $\frac{1}{3} (1,1,2)$, which is necessary contained in $(p = q = t = u = w = 0)$, and let $\varphi \colon (E \subset Y) \to (\msp \in X)$ be the Kawamata blowup.
We set $\mcC = \{p,q,u\}$ and $\Pi = \Pi (\mcC) \subset \mbP$.
Since $\msp_t, \msp_w \notin X$, we have $t^2 \in F_1, w^2 \in F_9$ and we can write
\[
F_1|_{\Pi} = t^2, \quad
F_9|_{\Pi} = w^2 + \alpha t^2 r + \beta t^2 s,
\]
where $\alpha, \beta \in \mbC$.
Thus,
\[
\Pi_X (\mcC) = \Pi \cap X = \Pi_X (p,q,t,u,w),
\]
and this consists of two $\frac{1}{6} (1,1,5)$ points and $\msp$.
In particular $D_p \cap D_q \cap D_u = \Pi_X (\mcC)$ is a finite set of points.
We have
\[
\ord_E (D_p) \ge \frac{1}{3}, \quad
\ord_E (D_q) \ge \frac{2}{3}, \quad
\ord_E (D_u) \ge \frac{2}{3},
\]
hence $N := - 8 \varphi^*K_X - \frac{2}{3} E$ is a nef divisor on $Y$ by Lemma \ref{lem:findnef}.
We have
\[
(N \cdot (-K_Y)^2) = 8 (-K_X)^3 - \frac{2}{3^3} \cdot \frac{3^2}{2} = - \frac{1}{15} < 0.
\]
By Lemma \ref{lem:exclsingNE}, $\msp$ is not a maximal center.

Let $\msp$ be a singular point of type $\frac{1}{6} (1,1,5)$.
After replacing $r$ and $s$, we may assume $\msp = \msp_s$.
We set $\mcC = \{p,q,r\}$.
We have
\[
\ivr_{\msp} (\mcC) = \frac{1}{6} = \wprod (\msp) (-K_X)^3.
\]
Since $\msp_t, \msp_u, \msp_v, \msp_w \notin X$, we may assume $t^2 \in F_1, u^2 \in F_3, v^2 \in F_6, w^2 \in F_9$ after possibly interchanging $F_3$ with $F_4$ and $F_6$ with $F_7$. 
Then, by setting $\Pi = \Pi (\mcC \cup \{s\})$ and by re-scaling $t, u, v, w$, we have
\[
F_1|_{\Pi} = t^2, \ 
F_3|_{\Pi} = u^2 + \alpha v t, \ 
F_6|_{\Pi} = v^2 + \beta w u, \ 
F_9|_{\Pi} = w^2,
\]
where $\alpha, \beta \in \mbC$.
This shows that $\Pi_X (\mcC \cup \{s\}) = \emptyset$ and $\msp$ is not a maximal center by Lemma \ref{lem:exclcrisingpt}.

Finally, let $\msp$ be a singular point of type $\frac{1}{5} (1,2,3)$ and let $\varphi \colon (E \subset Y) \to (\msp \in X)$ be the Kawamata blowup.
Replacing the coordinate $w$, we may assume $\msp = \msp_q$.
We write 
\[
\begin{split}
F_3 &= \lambda q^3 x + \mu q^2 r + \nu q^2 s + q f_{11} + f_{16}, \\
F_4 &= \lambda' q^3 x + \mu' q^2 r + \nu' q^2 s + q g_{11} + g_{16},
\end{split}
\] 
where $\lambda,\mu,\nu,\lambda', \mu', \nu' \in \mbC$ and $f_{11}, f_{16}, g_{11}, g_{16} \in \mbC [p,r,s,t,u,v,w]$ are homogeneous polynomials of the indicated degrees.

We first consider the case where $\mu \nu' - \nu \mu' \ne 0$. 
By replacing $r$ and $s$, we may assume that $\mu = \nu' = 1$ and $\lambda = \nu = \lambda' = \mu' = 0$.
We consider the initial weight at $\msp$
\[
\inwt (p, r, s, t, u, v, w) = \frac{1}{5} (1,1,1,2,3,4,5).
\]
Then $F_3^{\inwt} = q^2 r$ and $F_4^{\inwt} = q^2 s$, and this implies $\ord_E (D_r), \ord_E (D_s) \ge 6/5$.
Note that $\ord_E (D_p) \ge \inwt (x) = 1/5$.
We set $\mcC = \{p,r,s\}$ and $\Pi = \Pi (\mcC \cup \{q\})$.
By re-scaling $t, u, v, w$, we can write
\[
F_1|_{\Pi} = t^2, \ 
F_3|_{\Pi} = u^2 + \alpha v t, \ 
F_6|_{\Pi} = v^2 + \beta w u, \ 
F_9|_{\Pi} = w^2,
\]
where $\alpha, \beta \in \mbC$.
Hence $\Pi_X (\mcC \cup \{q\}) = \emptyset$.
Since $D_q$ is an apmle divisor, this implies that $D_p \cap D_r \cap D_s$ is a finite set of points (including $\msp$).
By Lemma \ref{lem:findnef}, $N := - \varphi^*K_X - \frac{1}{5} E$ is a nef divisor on $Y$.
We have
\[
(N \cdot (-K_Y)^2) = (-K_X)^3 - \frac{1}{5^3} (E^3) = \frac{1}{30} - \frac{1}{30} = 0,
\]
and this shows that $\msp$ is not a maximal center.

Next we consider the case where $\mu \nu' - \nu \mu' = 0$.
By replacing $r$ and $s$ suitably and by possibly interchanging $F_3$ and $F_4$, we may assume that 
\[
\begin{split}
F_3 &= q^3 p + q f_{11} + f_{16}, \\
F_4 &= q^2 s + q g_{11} + g_{16}.
\end{split}
\]
It is straightforward to see that $q^3 p$ is the unique monomial in $F_3$ with initial weight $1/5$, so that $\ord_E (D_p) \ge 6/5$.
Let $\mcL \subset |-6 K_X|$ be the pencil generated by the sections $r$ and $s$.
Since $\ord_E (D_r) = 1/5$ and $\ord_E (D_s) \ge 1/5$, a general member $L \in \mcL$ vanishes along $E$ to order $1/5$ so that $\tilde{L} \sim  - 6 \varphi^*K_X - \frac{1}{5} E$.
We have
\[
(-K_Y \cdot \tilde{D}_p \cdot \tilde{L}) = 6 (-K_X)^3 - \frac{\ord_E (D_p)}{5^2} \cdot (E^3) = \frac{1}{5} - \frac{\ord_E (D_p)}{6} \le 0
\]
since $\ord_E (p) \ge 6/5$.
By Lemma \ref{lem:exclsinginfini}, $\msp$ is not a maximal center and the proof is completed.
\end{proof}

\section{Birational superrigidity of cluster Fano $3$-folds}

In this section, we prove Theorem \ref{mainthm1}, which follow from Theorems \ref{thm:282G} and \ref{thm:282C} below.

\subsection{$\# 282$ by $\mathsf{G}_2^{(4)}$ format} \label{sec:proof282}

Let $X$ be a codimension $4$ prime Fano $3$-fold of numerical type $\# 282$ constructed in $\mathsf{G}_2^{(4)}$ format.
Then, by \cite[Example 5.5]{CD}, $X$ is defined by the following polynomials in $\mbP (1_p, 6_q, 6_r, 7_s, 8_t, 9_u, 10_v, 11_w)$.
\[
\begin{split}
F_1 &= t^2 - q v + s Q_9, \\
F_2 &= u t - q w + s( v + p^2 t), \\
F_3 &= t (v + p^2 t) - u Q_9 + q (q r + p^4 t), \\
F_4 &=(w + p^4 s) s - P_{12} q + u (u + p^2 s), \\
F_5 &= t w - u v + s (q r + p^4 t), \\
F_6 &=(q r + p^4 t) t - Q_9 w + v (v + p^2 t), \\
F_7 &=r s^2 - w u + t P_{12}, \\
F_8 &= P_{12} Q_9 - (v w + p^4 q w + p^2 u v + u q r + s t r - s t p^2), \\
F_9 &= r s (u + p^2 s) - v P_{12} + w (w + p^4 s).
\end{split}
\]
Here $P_{12}, Q_9 \in \mbC [p,q,r,s,t,u,v,w]$ are homogeneous polynomials of the indicated degree.
Recall that $(-K_X)^3 = 1/42$.

\begin{Thm} \label{thm:282G}
Let $X$ be a codimension $4$ Fano $3$-fold of numerical type $\#282$ constructed in $\mathsf{G}^{(4)}_2$ format.
Then $X$ is birationally superrigid.
\end{Thm}

\begin{proof}
By Proposition \ref{prop:282}, it remains to exclude the singular point $\msp \in X$ of type $\frac{1}{6} (1,1,5)$ as a maximal center.
The point $\msp$ corresponds to the unique solution solution of the equations
\[
p = s = t = u = v = w = F_3 = F_4 = 0,
\]
and we have $\msp = \msp_r$.
We set $\mcC = \{p,q\}$, $\Pi = \Pi (\mcC)$ and $\Gamma := \Pi_X (\mcC) = \Pi \cap X$.

We will show that $\Gamma$ is an irreducible and reduced curve.
We can write
\[
P_{12}|_{\Pi} = \lambda r^2, \quad
Q_9|_{\Pi} = \mu u,
\]
where $\lambda, \mu \in \mbC$.
By the quasi-smoothness of $X$ at $\msp$, we see that $\lambda, \mu \ne 0$.
Then we have
\begin{align*}
F_1|_{\Pi} &= t^2 + \mu s u, & F_4|_{\Pi} &= w s + u^2, & F_7|_{\Pi} &= r s^2 - w u + \lambda t r^2, \\
F_2|_{\Pi} &= u t + s v, & F_5|_{\Pi} &= t w - u v, & F_8|_{\Pi} &= \lambda \mu r^2 u - (v w + s t r), \\
F_3|_{\Pi} &= t v - \mu u^2, & F_6|_{\Pi} &= - \mu u w + v^2, & F_9|_{\Pi} &= r s u - \lambda v r^2 + w^2.
\end{align*}
We work on the open subset $U$ on which $w \ne 0$.
Then $\Gamma \cap U$ is isomorphic to the $\mbZ/11 \mbZ$-quotient of the affine curve
\[
(\lambda r^2 v + \mu^3 r v^6 - 1 = 0) \subset \mbA^2_{r,v}.
\]
It is straightforward to check that the polynomial $\lambda r^2 v + \mu^3 r v^6 -1$ is irreducible.
Thus $\Gamma \cap U$ is an irreducible and reduced affine curve.
It is also straightforward to check that
\[
\Gamma \cap (w = 0) = (p = q = w = 0) = \{\msp_r, \msp_s\}.
\]
This shows that $\Gamma$ is an irreducible and reduced curve.

Let $\varphi \colon (E \subset Y) \to (\msp \in X)$ be the Kawamata blowup and let $\tilde{\Delta}$ be the proper transform via $\varphi$ of a divisor or a curve on $X$.
We show that $\tilde{D}_p \cap \tilde{D}_q \cap E$ does not contain a curve.
Consider the initial weight
\[
\inwt (p, q, s, t, u, v, w) = \frac{1}{6} (1,6,1,2,3,4,5).
\]
We set $f_i = F_i^{\inwt} (p,q,1,s,t,u,v,w)$.
We have
\[
\begin{split}
f_4 &= (w + p^4)s - \lambda q + u (u + p^2 s), \\
f_7 &= s^2 + \lambda t, \\
f_8 &= \lambda \mu u - s t, \\
f_9 &= s (u + p^2 s) - \lambda v.
\end{split}
\]
Since $E$ is isomorphic to the subvariety
\[
(f_4 = f_7 = f_8 = f_9 = 0) \subset \mbP (1_p, 6_q, 1_s, 2_t, 3_u, 4_v, 5_w),
\]
it is straightforward to check that $\tilde{D}_p \cap \tilde{D}_q \cap E$ consists of finite set of points (in fact, $2$ points).
Thus we have $\tilde{D}_p \cdot \tilde{D}_q = \tilde{\Gamma}$ since $D_p \cdot D_q = \Gamma$.

We have 
\[
\tilde{D}_p \sim - \varphi^*K_X - \frac{1}{6} E, \quad
\tilde{D}_q \sim - 6 \varphi^*K_X - \frac{e}{6} E,
\]
for some integer $e \ge 6$ and hence
\[
(\tilde{D}_p \cdot \tilde{\Gamma}) 
= (\tilde{D}_p^2 \cdot \tilde{D}_q)
= \frac{1}{7} - \frac{e}{30} < 0.
\]
By \cite[Lemma 2.18]{OkadaII}, $\msp$ is not a maximal center.
\end{proof}

\subsection{$\# 282$ by $\mathsf{C}_2$ format}
Let $X$ be a codimension $4$ prime Fano $3$-fold of numerical type $\# 282$ constructed in $\mathsf{G}_2^{(4)}$ format.
Then, by \cite[Example 5.5]{CD}, $X$ is defined by the following polynomials in $\mbP (1_p, 6_q, 6_r, 7_s, 8_t, 9_u, 10_v, 11_w)$.
\[
\begin{split}
F_1 &= t R_8 - S_6 Q_{10} + s u, \\
F_2 &= t u - w S_6 + s v, \\
F_3 &= r S_6^2 - v R_8 + u^2, \\
F_4 &= t Q_{10} - S_6 P_{12} + s w, \\
F_5 &= r s S_6 - w R_8 + u Q_{10}, \\
F_6 &= r s^2 - P_{12} R_8 + Q_{10}^2, \\
F_7 &= r t S_6 - v Q_{10} + u w, \\
F_8 &= r s t - w Q_{10} + u P_{12}, \\
F_9 &= r t^2 - v P_{12} + w^2.
\end{split}
\]
Here $P_{12}, Q_{10}, R_8, S_6 \in \mbC [p,q,r,s,t,u,v,w]$ are homogeneous polynomials of the indicated degree.

\begin{Thm} \label{thm:282C}
Let $X$ be a quasi-smooth prime codimension $4$ Fano $3$-fold of numerical type $\#282$ constructed by $\mathsf{C}_2$ format.
We assume that $q \in S_6$.
Then $X$ is birationally superrigid.
\end{Thm}

\begin{proof}
By Proposition \ref{prop:282}, it remains to exclude the singular point $\msp$ of type $\frac{1}{6} (1,1,5)$ as a maximal center.

Replacing $q$, we may assume that $S_6 = q$.
The singular point $\msp$ corresponds to the solution of the equation
\[
p = s = t = u = v = w = S_6 = 0,
\]
and thus $\msp = \msp_r$.
We set $\mcC = \{p,q\}$ and $\Pi = \Pi (\mcC)$.

We will show that $\Gamma := \Pi \cap X$ is an irreducible and reduced curve.
We have $\Pi_X (\{p,q,r,s\}) = \emptyset$ (see the proof of Proposition \ref{prop:282}).
Hence $\Gamma \cap (s = 0) = \Pi_X (\{p,q,s\})$ does not contain a curve and it remains to show that $\Gamma \cap U_s$ is irreducible and reduced, where $U_s := (s \ne 0) \subset \mbP$ is the open subset.
We can write
\[
P_{12}|_{\Pi} = \lambda r^2, \ 
Q_{10}|_{\Pi} = \mu v, \ 
R_8|_{\Pi} = \nu t,
\]
for some $\lambda, \mu, \nu \in \mbC$, and we have $S_6|_{\Pi} = 0$.
We see that $t^2$ appears in $F_1$ (resp.\ $v^2$ appears in either $F_6$ or $F_7$) with non-zero coefficient since $\msp_t \notin X$ (resp.\ $\msp_v \notin X$), which implies that $\mu, \ne 0$.
Note that $F_i|_{\Pi} = F_i|_{\Pi} (r,s,t,u,v,w)$ is a polynomial in variables $r,s,t,u,v,w$ and we set $f_i = F_i|_{\Pi} (r,1,t,u,v,w)$.
Let $C \subset \mbA^5_{r,t,u,v,w}$ be the affine scheme defined by the equations
\[
f_1 = f_2 = \cdots = f_9 = 0.
\]
Then $\Gamma \cap U_s$ is isomorphic to the quotient of $C$ by the natural $\mbZ/7 \mbZ$-action.
We have
\begin{align*}
f_1 &= \nu t^2 + u, & f_2 &= t u + v, & f_3 &= - \nu t v + u^2, \\
f_4 &= \mu t v + w, & f_5 &= - \nu t w + \mu u v, & f_6 &= r - \lambda \nu r^2 t + \mu^2 v^2, \\
f_7 &= - \mu v^2 + u w, & f_8 &= r t - \mu v w + \lambda r^2 u, & f_9 &= r t^2 - \lambda r^2 v + w^2.
\end{align*}
By the equations $f_1 = 0$, $f_2 = 0$ and $f_4 = 0$, we have
\[
u = - \nu t^2, \ 
v = - t u = \nu t^3, \ 
w = - \mu t v = - \mu \nu t^4.
\]
By eliminating the variables $u, v, w$ and cleaning up the equations, $C$ is isomorphic to the hypersurface in $\mbA^2_{r, t}$ defined by
\[
r - \lambda \nu r^2 t + \mu^2 \nu^2 t^6 = 0,
\]
which is an irreducible and reduced curve since $\mu \nu \ne 0$, and so is $\Gamma \cap U_s$.
Thus $\Gamma$ is an irreducible and reduced curve.

Let $\varphi \colon (E \subset Y) \to (\msp \in X)$ be the Kawamata blowup.
We have $e := \ord_E (D_q) \ge 6/6$ and $\ord_E (D_p) = 1/6$ so that we have 
\[
\tilde{D}_q \sim - 6 \varphi^*K_X - \frac{e}{6} E = - 6 K_Y + \frac{6-e}{6} E, \quad
\tilde{D}_p \sim - \varphi^*K_X - \frac{1}{6} E = -K_Y.
\]
We show that $\tilde{D}_q \cap \tilde{D}_p \cap E$ does not contain a curve.
The Kawamata blowup $\varphi$ is realized as the weighted blowup at $\msp$ with the weight
\[
\inwt (p,q,s,t,u,v,w) = \frac{1}{6} (1,6,1,2,3,4,5).
\]
We have
\[
\begin{split}
F_4^{\inwt} &= - \lambda q r^2 + t (\mu v + h) + s w, \\
F_6^{\inwt} &= - \lambda \mu t r^2 + r s^2, \\
F_8^{\inwt} &= \lambda u r^2 + r s t, \\
F_9^{\inwt} &= - \lambda v r^2 + r t^2,
\end{split}
\]
where we define $h := Q_{10}^{\inwt} - \mu v$.
Note that $h$ is a linear combination of $u p, t p^2, s p^3, r p^4$ and thus $h$ is divisible by $p$. 
It follows that $E$ is isomorphic to the subscheme in $\mbP (1_p, 6_q, 1_s, 2_t, 3_u, 4_v, 5_w)$ defined by the equations
\[
\lambda q - t (\mu v + h) - s w = \lambda \mu t - s^2 = \lambda u + s t = \lambda v + t^2 = 0.
\]
It is now straightforward to check that $\tilde{D}_q \cap \tilde{D}_p \cap E = (p = q = 0) \cap E$ is a finite set of points (in fact, it consists of $2$ points).
This shows that $\tilde{D}_q \cdot \tilde{D}_p = \tilde{\Gamma}$ since $D_q \cdot D_p = \Gamma$.
We have
\[
(\tilde{D}_p \cdot \tilde{\Gamma}) = (\tilde{D}_p^2 \cdot \tilde{D}_q) = 6 (-K_X)^3 - \frac{e}{6^3} (E^3) = \frac{1}{7} - \frac{e}{30} < 0
\]
since $e \ge 6$.
By \cite[Lemma 2.18]{OkadaII}, $\msp$ is not a maximal center.
\end{proof}

\section{On further problems} \label{sec:discussion}

\subsection{Prime Fano $3$-folds with no projection centers}

We further investigate birational superrigidity of prime Fano $3$-folds of codimension $c$ with no projection centers for $5 \le c \le 9$.
There are only a few such candidates, which can be summarized as follows.

\begin{itemize}
\item In codimension $c \in \{5,7,8\}$, there is a unique candidate and it corresponds to smooth prime Fano $3$-folds of degree $2 c + 2$.
All of these Fano $3$-folds are rational (see \cite[Corollary 4.3.5 or \S 12.2]{IP}) and are not birationally superrigid.
\item In codimension $6$, there are $2$ candidates; one candidate corresponds to smooth prime Fano $3$-folds of degree $14$ which are birational to smooth cubic $3$-folds (see \cite{Takeuchi}, \cite{Isk}) and are not birationally superrigid, and the existence is not known for the other candidate which is $\# 78$ in the database. 
\item In codimension $9$, there is a unique candidate of smooth prime Fano $3$-folds of degree $20$.
However, according to the classification of smooth Fano $3$-folds there is no such Fano $3$-fold (see e.g.\ \cite[Theorem 0.1]{Takeuchi}).
\end{itemize}

It follows that, in codimension up to $9$, $\# 78$ is the only remaining unknown case for birational superrigidity (of general members). 

\begin{Question}
Do there exist prime Fano $3$-folds which correspond to $\# 78$?
If yes, then is a (general) such Fano $3$-fold birationally superrigid?
\end{Question}


In codimension $10$ and higher, there are a lot of candidates of Fano $3$-folds with no projection centers.
We expect that many of them are non-existence cases and that there are only a few birationally superrigid Fano $3$-folds in higher codimensions.

\begin{Question}
Is there a numerical type (in other words, Graded Ring Database ID)  $\# \mathsf{i}$ in codimension greater than $9$ such that a (general) quasi-smooth prime Fano $3$-fold of numerical type $\# \mathsf{i}$ is birationally superrigid?
\end{Question} 

\subsection{Classification of birationally superrigid Fano $3$-folds}

There are many difficulties in the complete classification of birationally superrigid Fano $3$-folds.
For example, we need to consider Fano $3$-folds which are not necessarily quasi-smooth or not necessarily prime, and also we need to understand subtle behaviors of birational superrigidity in a family, etc.
 
\begin{Question}
Is there a birationally superrigid Fano $3$-fold which is either of Fano index greater than $1$ or has a non-quotient singularity?
\end{Question}

\begin{Rem}
By recent developments \cite{Puk2}, \cite{Suz}, \cite{LZ}, it has been known that there exist birationally superrigid Fano varieties which have non-quotient singularities (see \cite{Puk2}, \cite{Suz}, \cite{LZ}) at least in very high dimensions.
On the other hand, only a little is known for Fano varieties of index greater than $1$ (cf.\ \cite{Puk1}) and there is no example of birationally superrigid Fano varieties of index greater than $1$.
\end{Rem}

We concentrate on quasi-smooth prime Fano $3$-folds.
Even in that case, it is necessary to consider those with a projection center, which are not treated in this paper.
Let $X$ be a general quasi-smooth prime Fano $3$-fold of codimension $c$.
Then the following are known.
\begin{itemize}
\item When $c = 1$, $X$ is birationally superrigid if and only if $X$ does not admit a type $\mathrm{I}$ projection center (see \cite{IM}, \cite{CPR}, \cite{CP}).
\item When $c = 2, 3$, $X$ is birationally superrigid if and only if $X$ is singular and admits no projection center (see \cite{IPuk}, \cite{OkadaI}, \cite{AZ}, \cite{AO}).
\end{itemize}

With these evidences, we expect the following.

\begin{Conj}
Let $X$ be a general quasi-smooth prime Fano $3$-fold of codimension at least $2$.
Then $X$ is birationally superrigid if and only if $X$ is singular and admits no projection centers.
\end{Conj}


\end{document}